\documentclass[reqno]{amsart}
\usepackage{amssymb}
\usepackage{amsmath}
\usepackage{amsfonts}
\usepackage[dvips]{graphicx}

\theoremstyle{plain}
\newtheorem{theorem}{Theorem}
\newtheorem{corollary}{Corollary}

\newtheorem{lemma}{Lemma}
\theoremstyle{definition}

\newtheorem{example}{Example}
\theoremstyle{remark}

\newtheorem{remark}{Remark}

\numberwithin{equation}{section}

\begin{document}
\title[Assouad dimensions of Moran sets]{Assouad dimensions of Moran sets
and Cantor-like sets}
\author{Wen-Wen Li}
\address{Department of Mathematics, East China Normal University, Shanghai
200241, P.~R. China}
\email{wenwen200309@163.com}
\author{Wen-Xia Li}
\address{Department of Mathematics, Shanghai Key Laboratory of PMMP, East
China Normal University, Shanghai 200241, P.~R. China}
\email{wxli@math.ecnu.edu.cn}
\author{Jun-Jie Miao}
\address{Department of Mathematics, East China Normal University, Shanghai
200241, P.~R. China}
\email{jjmiao@math.ecnu.edu.cn}
\author{Li-Feng Xi$^{\ast}$}
\address{Institute of Mathematics, Zhejiang Wanli University, Ningbo,
Zhejiang, 315100, P.~R. China}
\email{xilifengningbo@yahoo.com}

\begin{abstract}
We obtain the Assouad dimensions of Moran sets under suitable condition.
Using the homogeneous set introduced in \cite{Lv}, we also study the Assouad
dimensions of Cantor-like sets.
\end{abstract}

\thanks{$^{\ast }$ Corresponding author. Wen-Xia Li was supported by the
NNSF of China (no. 11271137). Jun-Jie Miao was partially supported by the
NNSF of China (no. 11201152), the Fund for the Doctoral Program of Higher
Education of China (no. 20120076120001) and NSF of Shanghai (no.
11ZR1410300). Li-Feng Xi was supported by the NNSF of China (no. 11371329),
NCET, NSF of Zhejiang Province (Nos. LR13A1010001, LY12F02011).}
\maketitle

%\maketitle

\section{Introduction}

Let $(X, d)$ be a metric space. We say $X$ is \textit{doubling} if there
exists an integer $N>0$ such that each ball in $X$ can be covered by $N$
balls of half the radius. Repeated applying this property, it gives that
there exist constants $b,c>0$ and $\alpha >0$ such that for all $r$ and $R$
with $0<r<R<b$, every ball $B(x,R)$ can be covered by $c(\frac{R}{r}%
)^{\alpha }$ balls of radius $r$. Let $N_{r,R}(X)$ denote the smallest
number of balls with radii $r$ which can cover a ball with radius $R$. The
\textit{Assouad dimension} of $X$, denoted by $\dim _{A}(X)$, is defined as
\begin{equation*}
\dim _{A}(X)=\inf \{\alpha \geq 0\text{ }|\text{ }\exists \text{ }b,c>0\text{
s.t. }N_{r,R}(X)\leq c(\frac{R}{r})^{\alpha }\text{ }\forall \text{ }%
0<r<R<b\},
\end{equation*}%
which was introduced by Assouad in the late 1970s \cite%
{Assouad1,Assouad2,Assouad3}. Now it plays a prominent role in the study of
quasiconformal mappings on $\mathbb{R}^{d}$, and we refer the readers to the
textbook \cite{Heinonen} and the survey paper \cite{Luukkainen} for more
details. It is well known that $\dim _{H}(X)\leq \dim _{P}(X)\leq \dim _{A}X,
$ where $\dim _{H}(\cdot ),\dim _{P}(\cdot )$ are Hausdorff and packing
dimensions respectively.

Suppose that $K$ is a compact subset of $X$ and $s$ is a non-negative real
number. We say $K$ is \textit{Ahlfors-David $s$-regular} if there exists a
Borel measure $\mu $ supported on $K$ and a constant $c\geq 1$ such that,
for all $x\in K$ and $0<r\leq |K|,$
\begin{equation}
c^{-1}r^{s}\leq \mu (B(x,r))\leq cr^{s},\text{ }  \label{AD}
\end{equation}%
where $B(x,r)$ is the closed ball centered at $x$ with radius $r$ and $%
|\cdot |$ denotes the diameter of set. Olsen \cite{Olsen} proved that for a
class of fractals with some flexible graph-directed construction, their
Assouad dimensions coincide with their Hausdorff and box dimensions. He also
pointed out that the fractals in \cite{Olsen} are Ahlfors-David regular. It
is well known that self-similar sets and self-conformal sets satisfying the
open set condition (\textbf{OSC}) are always Ahlfors-David regular, see~\cite%
{Mattila}. One advantage of such sets is that their dimensions coincide,
namely, for Ahlfors-David\emph{\ }$s$-regular set $K,$ $\dim _{A}K=\dim
_{H}K=\dim _{P}K=s$.

In general, it is difficult to compute the Assouad dimensions of sets which
are not Ahlfors-David regular. Mackay \cite{Mackay} calculated the Assouad
dimensions of two classes of self-affine fractals, namely, Bedford-McMullen
carpets~\cite{Bed} and Lalley-Gatzouras sets~\cite{Mcm},  and he also solved
the problem posed by Olsen \cite{Olsen}. Fraser \cite{Fraser} obtained
Assouad dimensions for certain classes of self-affine sets and
quasi-self-similar sets.

\smallskip

In this paper, we studied the Assouad dimension formula of Moran sets,
Cantor-like sets and homogeneous sets. Moran set was first studied by Moran
in \cite{Moran}, where most cases are not Ahlfors-David regular. First, we
recall the definition of Moran set.

Let $\{n_{k}(\geq 2)\}_{k\geq 1}$ be a sequence of positive integers$.$ For
each $k=0,1,2,\cdots$,let $D_{k}=\{u_{1}u_{2}\cdots u_{k}:1\leq u_{j}\leq
n_{j}\ $, $j\leq k\}$ be the set of words of length $k$, with $%
D_{0}=\{\emptyset \}$ containing only the empty word $\emptyset$. Let $%
D=\cup _{k=0}^{\infty }D_{k}$ be the set of all finite words. Suppose that $%
J\subset \mathbb{R}^{d}$ is a compact set with $\mbox{int}(J)\neq
\varnothing $ (we always write int($\cdot )$ for the interior of set). Let $%
\{\phi _{k}\}_{k\geq 1}$ be a sequence of positive real vectors where $\phi
_{k}=(c_{k,1},c_{k,2},\cdots ,c_{k,n_{k}})$ and $\Sigma
_{j=1}^{n_{k}}(c_{k,j})^{d}\leq 1$ for each $k\in \mathbb{N}$. We say the
collection $\mathcal{F}=\{J_{\mathbf{u}}:\mathbf{u}\in D\}$ of closed
subsets of $J$ fulfills the \textit{Moran structure} if it satisfies the
following Moran structure conditions (MSC):

(1) For each $\mathbf{u}\in D$, $J_{\mathbf{u}}$ is geometrically similar to
$J$, i.e., there exists a similarity $S_{\mathbf{u}}:\mathbb{R}%
^{d}\rightarrow \mathbb{R}^{d}$ such that $J_{\mathbf{u}}=S_{\mathbf{u}}(J)$%
. We write $J_{\emptyset }=J$ for empty word $\emptyset $.

(2) For all $k\in \mathbb{N}$ and $\mathbf{u}\in D_{k-1}$, the elements $J_{%
\mathbf{u}1}, J_{\mathbf{u}2},\cdots ,J_{\mathbf{u}n_{k}}$ of $\mathcal{F}$
are the subsets of $J_{\mathbf{u}}$ with disjoint interiors, ie., $\mbox{int}%
(J_{\mathbf{u}i})\cap \mbox{int}(J_{\mathbf{u}i^{\prime }})=\varnothing $
for $i\neq i^{\prime }$. Moreover, for all $1\leq i\leq n_{k} $,
\begin{equation*}
\frac{|J_{\mathbf{u}i}|}{|J_{\mathbf{u}}|}=c_{k,i}.
\end{equation*}%
We call $E=E(\mathcal{F})=\bigcap\nolimits_{k=1}^{\infty }\bigcup\nolimits_{%
\mathbf{u}\in D_{k}}J_{\mathbf{u}}$ a \textit{Moran set} determined by $%
\mathcal{F}$. For all $\mathbf{u}\in D_{k}$, the elements $J_{\mathbf{u}}$
are called \textit{\ $k$th-level basic sets} of $E$. Suppose the set $J$ and
the sequences $\{n_{k}\}$ and $\{\phi _{k}\}$ are given. We denote by $%
\mathcal{M}=\mathcal{M}(J,\{n_{k}\},\{\phi _{k}\})$ the class of the Moran
sets satisfying the MSC.

For any $k^{\prime }>k,$ let $s_{k,k^{\prime }}$ be the unique real solution
of the equation $\Delta _{k,k^{\prime }}(s)=1$, where
\begin{equation}  \label{eqns}
\Delta _{k,k^{\prime }}(s)=\prod\nolimits_{i=k+1}^{k^{\prime }}\left(
\sum\nolimits_{j=1}^{n_{i}}(c_{i,j})^{s}\right) .
\end{equation}
If the sequence $\{\sup\nolimits_{k}s_{k,k+m}\}_{m=1}^\infty$ converges, we
write
\begin{equation*}
s^{\ast \ast }=\lim_{m\rightarrow \infty }\left(
\sup\nolimits_{k}s_{k,k+m}\right).
\end{equation*}

In Section 2, we prove that the sequence $\{\sup\nolimits_{k}s_{k,k+m}%
\}_{m=1}^\infty$ is indeed convergent under the assumption $c_{\ast
}=\inf_{i,j}c_{i,j}>0.$ Furthermore, The following theorem indicates that
the limit is the Assouad dimension of Moran sets.

\begin{theorem}
\label{thmMS} Suppose that $\mathcal{M}=\mathcal{M}(J,\{n_{k}\},\{\phi
_{k}\})$ is a Moran class with $c_{\ast }=\inf_{i,j}c_{i,j}>0.$ Then, for
all $E\in \mathcal{M}$,
\begin{equation*}
\dim _{A}E=s^{\ast \ast}.
\end{equation*}
\end{theorem}

As an immediate consequence, we have the following corollary.

\begin{corollary}
Suppose that $\mathcal{M}=\mathcal{M}(J,\{n_{k}\},\{\phi _{k}\})$ is a Moran
class with $c_{\ast }=\inf_{i,j}c_{i,j}>0.$ Let $c_{k,1}=c_{k,2}=\cdots
=c_{k,n_{k}}=c_{k}\ $ for each $k\in\mathbb{N}$. Then, for all $F\in
\mathcal{M}$,
\begin{equation*}
\dim _{A}F=\lim_{m\rightarrow \infty }\left( \sup_{k}\frac{\log (n_{k}\cdots
n_{k+m})}{-\log (c_{k}\cdots c_{k+m})}\right) .
\end{equation*}
\end{corollary}

Let $s_{\ast }$ and $s^{\ast }$ be the upper and lower limits of the
sequence $\{s_{0,m}\}_{m=1}^\infty$, that is,
\begin{equation*}
s_{\ast }=\underline{\lim }_{m\rightarrow \infty }s_{0,m} \text{ and }
s^{\ast }=\overline{\lim }_{m\rightarrow \infty }s_{0,m}.
\end{equation*}%
It was shown in \cite{Hua,HuaLi,Wen1,Wen2} that, for all $E\in \mathcal{M}$
with $c_{\ast }>0,$
\begin{equation*}
\dim _{H}E=s_{\ast }\text{ and }\dim _{P}E=s^{\ast }.
\end{equation*}%
In next example, we will construct a Moran set to satisfy $\dim _{H}E<\dim
_{P}E<\dim _{A}E.$

Note that it is also a counter-example to the conclusion in \cite{Li}.
Hereby, Theorem~\ref{thmMS} corrects their main conclusion.

\begin{example}
Let $\{p_{i}\}_{i}$ be an increasing sequence of integers such that $%
p_{i+1}-p_{i}>i$ for all $i$ and
\begin{equation*}
\lim_{i\rightarrow \infty }\frac{p_{i-1}}{p_{i}-p_{i-1}}=\lim_{i\rightarrow
\infty }\frac{i}{p_{i}-p_{i-1}}=0.
\end{equation*}%
Let $J=[0,1],n_{k}\equiv 2$ and
\begin{equation*}
c_{k,1}=c_{k,2}=\left\{
\begin{array}{ll}
1/4 & \text{if }k\in \lbrack p_{i}+1,p_{i}+i]\text{ for some\ }i\in \mathbb{N%
}, \\
1/8 & \text{if }k\in \lbrack p_{i}+i+1,p_{i+1}]\text{ for some\ even }i\in
\mathbb{N}, \\
1/16 & \text{if }k\in \lbrack p_{i}+i+1,p_{i+1}]\text{ for some\ odd }i\in
\mathbb{N}.%
\end{array}%
\right.
\end{equation*}%
Then we have $s_{\ast }=\frac{1}{4},$ $s^{\ast }=\frac{1}{3},$ $s^{\ast \ast
}=\frac{1}{2}.$ Clearly, for all $E\in \mathcal{M}(J,\{n_{k}(\equiv
2)\},\{(c_{k,1},c_{k,2})\})$, the dimensions inequality strictly holds, that
is ,
\begin{equation*}
\dim _{H}E=\frac{1}{4}<\dim _{P}E=\frac{1}{3}<\dim _{A}E=\frac{1}{2}.
\end{equation*}
\end{example}

Suppose $\{a_{n}\}$ is a sequence of positive numbers with $%
\sum\nolimits_{n}a_{n}<\infty .$ Given sequences $\{c_{k}\}_{k\geq 1\text{ }
} $and $\{n_{k}\}_{k\geq 1}$ such that $c_{k}\in (0,1)$ and $n_{k}\in
\mathbb{N}\cap \lbrack 2,\infty )$ for all $k\in\mathbb{N}.$ We always
assume that $c_{\ast }=\inf_{k}c_{k}>0.$ Let $I$ be the initial set such
that $\mbox{int}(I)\neq \varnothing $. For each $i_{1}\cdots i_{k-1}\in
D_{k-1},$ suppose that $I_{i_{1}\cdots i_{k-1}1}$, $I_{i_{1}\cdots i_{k-1}2}$%
, $\cdots$, $I_{i_{1}\cdots i_{k-1}n_{k}}\subset I_{i_{1}\cdots i_{k-1}}$
are geometrically similar to $I_{i_{1}\cdots i_{k-1}}$ such that
\begin{equation*}
c_{k}(1-a_{k})\leq \frac{|I_{i_{1}\cdots i_{k-1}j}|}{|I_{i_{1}\cdots
i_{k-1}}|}\leq c_{k}(1+a_{k}), \quad j=1,2,\cdots,n_k,
\end{equation*}%
where the interiors of $I_{i_{1}\cdots i_{k-1}j}$ are pairwise disjoint. We
call
\begin{equation*}
K=\bigcap \nolimits_{k=1}^{\infty }\bigcup\nolimits_{i_{1}\cdots i_{k}\in
D_{k}}I_{i_{1}\cdots i_{k}}
\end{equation*}
a \textit{Cantor-like set}, and we write $\mathcal{C}(I,\{c_{k}\}_{k},%
\{n_{k}\}_{k},\{a_{k}\}_{k})$ for the collection of such sets.

\begin{remark}
Cantor-like sets may not be Moran sets.
\end{remark}

\begin{theorem}
Suppose that $K\in \mathcal{C}(I,\{c_{k}\}_{k},\{n_{k}\}_{k},\{a_{k}\}_{k})$
is a Cantor-like set$.$ Then
\begin{equation*}
\dim _{A}K=\lim_{m\rightarrow \infty }\left( \sup_{k}\frac{\log (n_{k}\cdots
n_{k+m})}{-\log (c_{k}\cdots c_{k+m})}\right) .
\end{equation*}
\end{theorem}

In fact, for Ahlfors-David regular set, using (\ref{AD}), there exist
constants$\ 0<\eta <1\leq \lambda $ and $1<\delta \leq \Delta <\infty $ such
that, for all $x,x^{\prime }\in K$ and $r\leq |K|,$
\begin{equation}
\lambda ^{-1}\leq \frac{\mu (B(x,r))}{\mu (B(x^{\prime },r))}\leq \lambda ,
\label{self1}
\end{equation}%
\begin{equation}
\delta \leq \frac{\mu (B(x,r))}{\mu (B(x,\eta r))}\leq \Delta .\text{ }
\label{***}
\end{equation}%
It follows from (\ref{***}) that the measure $\mu $ and set $K$ are
doubling, and $K$ is uniformly perfect \cite{Lv}. We say a compact subset $K$
of $X$ is \textit{homogeneous} if there exists a Borel measure $\mu $
supported on $K$ satisfying $(\ref{self1})$ and $(\ref{***}),$ and we refer
the readers  to~\cite{Lv} for details.

\begin{remark}
All Ahlfors-David regular sets are homogeneous, but homogeneous sets may not
be Ahlfors-David regular.
\end{remark}

Given a point $x\in K,$ we write
\begin{equation}
\alpha _{x}(r)=\frac{\log \mu (B(x,r))}{\log r},  \label{E:equ0}
\end{equation}%
for $0<r\leq |K|$. Here $\alpha _{x}(r)$ is like the function with respect
to pointwise dimension of measure.

Given $\epsilon>0$, we write
\begin{equation*}
\Omega =\{g(r):(0,\varepsilon )\rightarrow \mathbb{R}^{+}|\text{ }%
0<\inf_{r<\varepsilon }g(r)\leq \sup_{r<\varepsilon }g(r)<\infty \}.
\end{equation*}%
For each $g\in \Omega ,$ we focus on the behavior of function $g(r)$ as $r$
tends to $0.$ If a mapping $h\in\Omega $ satisfies that, for all $%
r<\varepsilon ,$%
\begin{equation}
\left\vert h(r)-g(r)\right\vert \leq C|\log r|^{-1}  \label{E:equ}
\end{equation}%
for some constant $C,$ we say $h$ and $g$ are \textit{equivalent}, denoted
by $g\sim h$, and we write equivalence class $[g]=\{h:g\sim h\}.$ By the
result of \cite{Lv}, we have $\alpha _{x}(r)\in \Omega .$ Notice that $%
\alpha _{x}(r)\sim \alpha _{x^{\prime }}(r)$ by (\ref{self1}), we use $h(r)$
to denote any function in the equivalence class $[\alpha _{x}(r)]$ with $%
x\in K$, and $h(r)$ is called a \textit{scale function} of $K.$

\begin{remark}
For Ahlfors-David $s$-regular set$,$ we can take $h(r)\equiv s.$
\end{remark}

It is easy to check $\dim _{H}K=\liminf_{r\rightarrow 0}h(r)$ and $\dim
_{P}K=\limsup_{r\rightarrow 0}h(r)$, see~\cite{Lv}. Similarly, scale
functions also play an important role in the Assouad dimension formula of
homogeneous sets.

\begin{theorem}
Suppose $K$ is homogeneous with a scale function $h(r)$. Then
\begin{equation*}
\dim _{A}K=\lim_{\rho \rightarrow 0}\left( \sup_{R}\left\vert \frac{h(R)\log
R-h(\rho R)\log (\rho R)}{\log \rho }\right\vert \right) .
\end{equation*}
\end{theorem}

\begin{remark}
Suppose that $h(r)$ is defined on $(0,\varepsilon ).$ Using $(\ref{***}),$
we can obtain that for all $\varepsilon _{1},\varepsilon _{2}\leq
\varepsilon ,$
\begin{equation*}
\lim_{\rho \rightarrow 0}\sup_{R<\varepsilon _{1}}\psi (R,\rho )=\lim_{\rho
\rightarrow 0}\sup_{R<\varepsilon _{2}}\psi (R,\rho ),
\end{equation*}%
where $\psi (R,\rho )=\left\vert \frac{h(R)\log R-h(\rho R)\log (\rho R)}{%
\log \rho }\right\vert .$
\end{remark}

For each Cantor-like set $K\in \mathcal{C}(I,\{c_{k}\}_{k},\{n_{k}\}_{k},%
\{a_{k}\}_{k}),$ using the approach in \cite{Lv}, it is clear that $K$ is
homogeneous with a scale function
\begin{equation*}
h(r)=\frac{\log n_{1}\cdots n_{k}}{-\log c_{1}\cdots c_{k}}\quad \text{ for }%
c_{1}\cdots c_{k}|I|<r\leq c_{1}\cdots c_{k-1}|I|.
\end{equation*}%
Therefore Theorem 2 follows immediately from Theorem 3.

\begin{remark}
Using the result in \cite{Lv}, for every Cantor-like set $K$ as above, we
have
\begin{equation*}
\dim _{H}K=\underline{\lim }_{k\rightarrow \infty }\frac{\log n_{1}\cdots
n_{k}}{-\log c_{1}\cdots c_{k}} , \quad \dim _{P}K=\overline{\lim }%
_{k\rightarrow \infty }\frac{\log n_{1}\cdots n_{k}}{-\log c_{1}\cdots c_{k}}%
.
\end{equation*}
\end{remark}

For the rest of the paper, we will prove Theorem 1 and Theorem 3 in Section
2 and Section 3 respectively.

\bigskip

\section{Assouad Dimension of Moran Set}

Suppose that $\mathcal{M}=\mathcal{M}(J,\{n_{k}\},\{\phi _{k}\})$ where $%
\phi _{k}=(c_{k,1},c_{k,2},\cdots ,c_{k,n_{k}})$, $k=1,2,\cdots.$ Without
loss of generality, we assume that $|J|=1.$

For each word $\mathbf{u}=u_{1}u_{2}\cdots u_{k}\in D_{k}$, we write $|%
\mathbf{u}|(=k) $ for the length of $\mathbf{u}$. Given $k,k^{\prime }\in
\mathbb{N}$, we write
\begin{equation*}
D_{k,k^{\prime }}= \{\mathbf{v}=v_{k}\cdots v_{k^{\prime }}:1\leq v_{j}\leq
n_{j}\text{ for }k\leq j\leq k^{\prime }\},
\end{equation*}%
for $k\leq k^{\prime}$, otherwise, $D_{k,k^{\prime }}=\{\emptyset \}$. Note
that $D_{1,k}=D_{k}$. For $\mathbf{v}=v_{k}\cdots v_{k^{\prime }}\in
D_{k,k^{\prime }},$ we write
\begin{equation*}
c_{\mathbf{v}}=c_{k,v_{k}}\cdots c_{k^{\prime },v_{k^{\prime }}},
\end{equation*}%
with $c_{\mathbf{\emptyset }}=1$. For $\mathbf{u}=u_{1}u_{2}\cdots
u_{k-1}\in D_{k-1}$ and $\mathbf{v}=v_{k}v_{k+1}\cdots v_{k^{\prime }}\in
D_{k,k^{\prime }}$, we write
\begin{equation*}
\mathbf{u}\ast \mathbf{v}=u_{1}u_{2}\cdots u_{k-1}v_{k}v_{k+1}\cdots
v_{k^{\prime }}\in D_{k^{\prime }}.
\end{equation*}
For $\mathbf{v}\in D_{k,k^{\prime }}$, we denote by $\mathbf{v}^{-}$ the
word obtained by deleting the last letter of $\mathbf{v}$. Note that $%
\mathbf{v}^{-}=\emptyset $ (the empty word) if $k=k^{\prime }-1$. Given $%
\mathbf{u}\in D$, for $0<\delta <c_{\ast }$, we write
\begin{equation}
\mathcal{A}_{\mathbf{u}}(\delta )=\{\mathbf{u}\ast \mathbf{v}\in D:\text{ }%
c_{\mathbf{v}}\leq \delta <c_{\mathbf{v}^{-}}\}.  \label{2.1}
\end{equation}%
For $\mathbf{u}=\emptyset ,$ we write $\mathcal{A}(\delta )$ for $\mathcal{A}%
_{\emptyset }(\delta ).$

Let $\Lambda =\{u_{1}u_{2}\cdots u_{k}\cdots :u_{1}u_{2}\cdots u_{k}\in
D_{k} $ for all $k\}$ be the symbolic system composed of infinite words.
Given a word $\mathbf{i=}i_{1}\cdots i_{n}\in D_{n},$ we call
\begin{equation*}
\lbrack \mathbf{i}]=\{u_{1}\cdots u_{n}\cdots \in \Lambda :u_{1}\cdots
u_{n}=i_{1}\cdots i_{n}\}
\end{equation*}
the \textit{cylinder} with respect to the word $\mathbf{i}.$

\begin{lemma}
\label{lem1} Given $\mathbf{u}\in D,$ we have
\begin{equation}
1=\sum\limits_{\mathbf{u}\ast \mathbf{v}\in \mathcal{A}_{\mathbf{u}}(\delta
)}\frac{(c_{\mathbf{v}})^{s}}{\prod_{p=|\mathbf{u}|+1}^{|\mathbf{u}|+|%
\mathbf{v}|}\sum\limits_{q=1}^{n_{p}}c_{p,q}^{s}}.  \label{xixi}
\end{equation}
\end{lemma}

\begin{proof}
Fix $\mathbf{u}\in D,$ we have a probability measure $\mu $ supported on $[%
\mathbf{u}]$ such that
\begin{equation*}
\mu ([\mathbf{u}\ast \mathbf{v}])=\frac{(c_{\mathbf{v}})^{s}}{\prod_{p=|%
\mathbf{u}|+1}^{|\mathbf{u}|+|\mathbf{v}|}\sum%
\limits_{q=1}^{n_{p}}c_{p,q}^{s}}\text{ for all }\mathbf{u}\ast \mathbf{v}%
\in D.
\end{equation*}%
Since $[\mathbf{u}]=\bigcup\limits_{\mathbf{u}\ast \mathbf{v}\in \mathcal{A}%
_{\mathbf{u}}(\delta )}[\mathbf{u}\ast \mathbf{v}]$ is a disjoint union, we
obtain
\begin{equation*}
1=\sum\limits_{\mathbf{u}\ast \mathbf{v}\in \mathcal{A}_{\mathbf{u}}(\delta
)}\mu ([\mathbf{u}\ast \mathbf{v}])=\sum\limits_{\mathbf{u}\ast \mathbf{v}%
\in \mathcal{A}_{\mathbf{u}}(\delta )}\frac{(c_{\mathbf{v}})^{s}}{\prod_{p=|%
\mathbf{u}|+1}^{|\mathbf{u}|+|\mathbf{v}|}\sum%
\limits_{q=1}^{n_{p}}c_{p,q}^{s}}.
\end{equation*}
\end{proof}

The following lemma can be obtained directly by using Lemma 9.2 in \cite%
{Falconer}.

\begin{lemma}
\label{lem2} \label{lemma1} Suppose $c_{\ast }>0$. Then there exits a
positive integer $l$ such that for all $0<\delta <c_{\ast },$ $\mathbf{u}\in
D$ and $x\in E\cap J_{\mathbf{u}},$ we have
\begin{equation*}
\sharp \{\mathbf{u}\ast \mathbf{v}\in \mathcal{A}_{\mathbf{u}}(\delta )\text{
}|\text{ }B(x,c_{\mathbf{u}}\delta )\cap J_{\mathbf{u}\ast \mathbf{v}}\neq
\varnothing \}\leq l.
\end{equation*}%
In particular, if $\mathbf{u}$ is the empty word, we have
\begin{equation*}
\sharp \{\mathbf{v}\in \mathcal{A}(\delta )\text{ }|\text{ }B(x,\delta )\cap
J_{\mathbf{v}}\neq \varnothing \}\leq l.
\end{equation*}
\end{lemma}

\medskip

\begin{lemma}
Suppose $c_{\ast }>0$. Let $s_{k,k+m}$ be defined by $(\ref{eqns})$. Then
the sequence $\{\sup_{k}s_{k,k+m}\}_{m=1}^\infty$ is convergent.
\end{lemma}

\begin{proof}
Suppose $E\subset \mathbb{R}^{d}.$ We denote by $\mathcal{L}$ the Lebesgue
measure on $\mathbb{R}^{d}.$ Recall that $c_{\ast }>0$ and $n_{k}\geq 2$.
Since for each $\mathbf{u}\in D_{k-1},$
\begin{equation*}
\text{int}(J_{\mathbf{u}\ast i})\cap \text{int}(J_{\mathbf{u}\ast
j})=\varnothing ,
\end{equation*}
for all $i\neq j\leq n_{k}$, we have%
\begin{equation*}
\sum\nolimits_{i=1}^{n_{k}}\mathcal{L}(\text{int(}J_{\mathbf{u}\ast i}\text{%
))}\leq \mathcal{L}(\text{int(}J_{\mathbf{u}}\text{)),}
\end{equation*}%
that is, $\sum\nolimits_{i=1}^{n_{k}}c_{k,i}^{d}\leq 1$. It implies%
\begin{equation}
\max_{i}c_{k,i}\leq (1-c_{\ast }^{d})^{1/d}\text{ and }\sup_{k}n_{k}\leq
(c_{\ast })^{-d}.  \label{aa}
\end{equation}

For every $m$, we write
\begin{equation*}
\theta _{m}=\sup_{k}s_{k,k+m}.
\end{equation*}%
Fix an integer $m\in \mathbb{N}$. Let $s>\theta _{m}$. For each $n\in
\mathbb{Z}\cap \lbrack 0,m-1]$, we have
\begin{equation*}
\Delta _{t,t+pm+n}(s)=\left( \prod\nolimits_{i=0}^{p-1}\Delta
_{t+im,t+(i+1)m}(s)\right) \cdot \Delta _{t+pm,t+pm+n}(s).
\end{equation*}%
Hence, by (\ref{aa}), $\Delta _{t+pm,t+pm+n}(s)\leq (\sup_{k}n_{k})^{n}\leq
(c_{\ast })^{-nd}$ and
\begin{eqnarray*}
\Delta _{t+im,t+(i+1)m}(s) &\leq &\Delta _{t+im,t+(i+1)m}(\theta _{m})\cdot
(\sup_{k,i}c_{k,i})^{s-\theta _{m}} \\
&\leq &1\cdot (1-c_{\ast }^{d})^{(s-\theta _{m})/d} \\
&=&(1-c_{\ast }^{d})^{(s-\theta _{m})/d}.
\end{eqnarray*}%
Therefore, for all $t\in \mathbb{N}$, we have
\begin{equation*}
\Delta _{t,t+pm+n}(s)\leq (1-c_{\ast }^{d})^{p(s-\theta _{m})/d}(c_{\ast
})^{-nd}
\end{equation*}%
which means there exists an integer $p_{0}(s)$ such that for all $p\geq
p_{0}(s),$
\begin{equation*}
\Delta _{t,t+pm+n}(s)\leq 1,
\end{equation*}%
that is,
\begin{equation*}
s_{t,t+pm+n}\leq s,
\end{equation*}%
for all $p\geq p_{0}(s)$ and $t\geq 0$. Hence
\begin{equation*}
\overline{\lim }_{p\rightarrow \infty }\theta _{pm+n}=\overline{\lim }%
_{p\rightarrow \infty }\sup\nolimits_{t}s_{t,t+pm+n}\leq s.
\end{equation*}%
Since it holds for all $s>\theta _{m}$, we obtain $\overline{\lim }%
_{p\rightarrow \infty }\theta _{pm+n}\leq \theta _{m}$. Thus%
\begin{equation*}
\overline{\lim }_{p\rightarrow \infty }\theta _{pm+n}\leq \inf_{m}\theta
_{m}\leq \underline{\lim }_{m\rightarrow \infty }\theta _{m},
\end{equation*}%
which implies that $\lim_{m\rightarrow \infty }\theta _{m}$ exists.$\ $
\end{proof}

\medskip

\begin{proof}[Proof of Theorem 1]
$\ $

We first prove that $s^{\ast \ast }$ is an upper bound of $\dim _{A}E$. It
suffices to verify that the inequality $\dim _{A}E\leq s$ holds for all $%
s>s^{\ast \ast }$.

Since $s>\lim_{m\rightarrow \infty}(\sup_{k}s_{k,k+m})$, there exists a
positive integer $N$ such that, for all $m>N$, we have $s>s_{k,k+m}$.
Therefore, for all $m>N$
\begin{equation}
\prod\limits_{i=k+1}^{k+m}\left( \sum\limits_{j=1}^{n_{i}}c_{i,j}^{s}\right)
\leq \prod\limits_{i=k+1}^{k+m}\left(
\sum\limits_{j=1}^{n_{i}}c_{i,j}^{s_{k,k+m}}\right) =1.  \label{2.2}
\end{equation}

Fix a word $\mathbf{i}\in D$ and $\delta \in (0,c_{\mathbf{i}^{-}}).$ The
fact that $c_{\ast }>0$ implies that the sequence $\{n_{k}\}$ is bounded,
say $\varpi >1,$ that is, $n_{k}\leq \varpi, k=1,2,\cdots.$ Thus, for all $%
0<m\leq N$,
\begin{equation}
\prod\limits_{i=k+1}^{k+m}\left( \sum\limits_{j=1}^{n_{i}}c_{i,j}^{s}\right)
\leq \varpi ^{N}.  \label{ha}
\end{equation}%
By (\ref{2.2}) and (\ref{ha}), we have, for all $\mathbf{j}\in D_p$,
\begin{equation}
\prod_{p=|\mathbf{i}|+1}^{|\mathbf{i}|+|\mathbf{j}|}\sum%
\limits_{q=1}^{n_{p}}c_{p,q}^{s}\leq \varpi ^{N}.  \label{we}
\end{equation}
Combining Lemma~\ref{lem1} with (\ref{we}), we have
\begin{align*}
1& =\sum\limits_{\mathbf{i}\ast \mathbf{j}\in \mathcal{A}_{\mathbf{i}%
}(\delta )}\frac{(c_{\mathbf{j}})^{s}}{\prod_{p=|\mathbf{i}|+1}^{|\mathbf{i}%
|+|\mathbf{j}|}\sum\limits_{q=1}^{n_{p}}c_{p,q}^{s}} \\
& \geq \varpi ^{-N}\sum\limits_{\mathbf{i}\ast \mathbf{j}\in \mathcal{A}_{%
\mathbf{i}}(\delta )}(c_{\mathbf{j}})^{s} \\
& \geq \varpi ^{-N}\sum\limits_{\mathbf{i}\ast \mathbf{j}\in \mathcal{A}_{%
\mathbf{i}}(\delta )}(c_{\ast }c_{\mathbf{j}^{-}})^{s} \\
& \geq (\varpi ^{-N}c_{\ast }^{s})\cdot \delta ^{s}\cdot \sharp \mathcal{A}_{%
\mathbf{i}}(\delta ).
\end{align*}%
It follows that
\begin{equation}
\sharp \mathcal{A}_{\mathbf{i}}(\delta )\leq \frac{\varpi ^{N}}{c_{\ast
}^{s}\delta ^{s}}  \label{2.3}
\end{equation}%
for all $\mathbf{i}\in D$ and all $0<\delta <c_{\ast }$.

Fix a point $x\in E$ and $r,R$ with $0<r<R$. Since $E$ is doubling, without
loss of generality, we may assume that
\begin{equation*}
0<r<c_{\ast }R<R<c_{\ast }.
\end{equation*}%
It is clear that
\begin{equation}
B(x,R)\cap E\subset \bigcup\limits_{\mathbf{i}\in \mathcal{A}(R),B(x,R)\cap
J_{\mathbf{i}}\neq \varnothing }J_{\mathbf{i}}\cap E.  \label{2.4}
\end{equation}
For each $\mathbf{i}\in \mathcal{A}(R)$ with $B(x,R)\cap J_{\mathbf{i}}\neq
\varnothing $, we have
\begin{equation*}
J_{\mathbf{i}}\cap E\subseteq \bigcup\limits_{\mathbf{i}\ast \mathbf{j}\in
\mathcal{A}_{\mathbf{i}}(r/R)}J_{\mathbf{i}\ast \mathbf{j}}.
\end{equation*}
Now taking $x_{\mathbf{i},\mathbf{j}}\in J_{\mathbf{i}\ast \mathbf{j}}\cap E$%
, we have
\begin{equation*}
J_{\mathbf{i}\ast \mathbf{j}}\subseteq B(x_{\mathbf{i},\mathbf{j}},r),
\end{equation*}
due to $c_{\mathbf{i}\ast \mathbf{j}}=c_{\mathbf{i}}c_{\mathbf{j}}\leq
R\cdot r/R=r$ and $|J|=1.$

Thus by (\ref{2.4}), we obtain that
\begin{equation}
B(x,R)\cap E\subset \bigcup\limits_{\substack{ \mathbf{i}\in \mathcal{A}(R)
\\ B(x,R)\cap J_{\mathbf{i}}\neq \varnothing }}\bigcup\limits_{\mathbf{i}%
\ast \mathbf{j}\in \mathcal{A}_{\mathbf{i}}(r/R)}B(x_{\mathbf{i},\mathbf{j}%
},r).  \label{2.6}
\end{equation}%
By (\ref{2.3}) and Lemma \ref{lem2}, we have
\begin{align*}
N_{r,R}(E)& \leq \sum\limits_{\substack{ \mathbf{i}\in \mathcal{A}(R)  \\ %
B(x,R)\cap J_{\mathbf{i}}\neq \varnothing }}\sharp \mathcal{A}_{\mathbf{i}%
}(r/R) \\
& \leq \sum\limits_{\substack{ \mathbf{i}\in \mathcal{A}(R)  \\ B(x,R)\cap
J_{\mathbf{i}}\neq \varnothing }}\frac{\varpi ^{N}}{c_{\ast }^{s}}\left(
\frac{R}{r}\right) ^{s} \\
& \leq \frac{\varpi ^{N}}{c_{\ast }^{s}}\left( \frac{R}{r}\right) ^{s}\cdot
\sharp \{\mathbf{i}\in \mathcal{A}(R):B(x,R)\cap J_{\mathbf{i}}\neq
\varnothing \} \\
& \leq \frac{l\varpi ^{N}}{c_{\ast }^{s}}\left( \frac{R}{r}\right) ^{s}.
\end{align*}%
Hence $s$ is an upper bound, and the arbitrariness implies that
\begin{equation*}
\dim _{A}E\leq s^{\ast \ast }.
\end{equation*}

\medskip

For the rest of the proof, we will verify that $s^{\ast \ast }$ is also a
lower bound.

Since $s^{\ast \ast }$ is the limit of $\{\sup_{k}s_{k,k+m}\}$, there exists
a sequence $\{(m_{k},m_{k}^{\prime })\}_{k=1}^{\infty }$ of integer pairs
with $(m_{k}^{\prime }-m_{k})$ tending to $\infty$ such that
\begin{equation*}
\lim_{k\rightarrow \infty }s_{m_{k},m_{k}^{\prime }}=s^{\ast \ast }.
\end{equation*}

Arbitrarily choose $s<s^{\ast \ast }$. Without loss of generality, we assume
that, for all $k\in \mathbb{N}$.
\begin{equation*}
s_{m_{k},m_{k}^{\prime }}>s.
\end{equation*}%
Hence, it is clear that
\begin{equation*}
\Delta _{m_{k},m_{k}^{\prime }}(s)>\Delta _{m_{k},m_{k}^{\prime
}}(s_{m_{k},m_{k}^{\prime }})=1.
\end{equation*}%
Fix an integer $k$, we have
\begin{equation}
\sum\limits_{\mathbf{j}\in D_{m_{k}+1,m_{k}^{\prime }}}c_{\mathbf{j}%
}^{s}=\Delta _{m_{k},m_{k}^{\prime }}(s)>1.  \label{2.7}
\end{equation}%
For each $p\in \mathbb{N}\cup \{0\}$, let
\begin{equation}
\mathcal{B}_{p,k}=\{\mathbf{j}\in D_{m_{k}+1,m_{k}^{\prime }}:2^{-p-1}<c_{%
\mathbf{j}}\leq 2^{-p}\},  \label{test}
\end{equation}%
and we write
\begin{equation*}
p_{k}=\min \{p:\mathcal{B}_{p,k}\neq \varnothing \}.
\end{equation*}%
Since
\begin{equation*}
2^{-p-1}\leq c_{\mathbf{j}}\leq (1-c_{\ast }^{d})^{(m_{k}^{\prime
}-m_{k})/d},
\end{equation*}%
it is obvious that the sequence $\{p_k\}$ tends to infinity, that is,
\begin{equation*}
\lim_{k}p_{k}=\infty ,
\end{equation*}%
Thus by (\ref{2.7}) and (\ref{test}), we obtain that
\begin{equation}
\sum\limits_{p=0}^{\infty }\sharp \mathcal{B}_{p,k}2^{-ps}>1.  \label{fff}
\end{equation}
Hence, for any $\varepsilon >0$, there exists an integer $q_{k}(\geq p_{k})$
such that
\begin{equation}
2^{-\varepsilon q_{k}}(1-2^{-\varepsilon })\leq \sharp \mathcal{B}%
_{q_{k},k}(2^{-q_{k}})^{s},  \label{q_k}
\end{equation}%
otherwise
\begin{equation*}
\sum\limits_{p=0}^{\infty }\sharp \mathcal{B}_{p,k}2^{-ps}<\sum%
\limits_{p=0}^{\infty }2^{-\varepsilon p}(1-2^{-\varepsilon })=1,
\end{equation*}%
which contradicts (\ref{fff}). Since $p_{k}$ tends to $\infty $ and $%
q_{k}\geq p_{k},$ we have
\begin{equation*}
\lim_{k}q_{k}=\infty .
\end{equation*}

Given $\mathbf{i}\in D_{m_{k}},$ we take
\begin{equation*}
R_{k}=|J_{\mathbf{i}}|\text{\ and }r_{k}=\min_{\mathbf{j}\in \mathcal{B}%
_{q_{k},k}}|J_{\mathbf{i}\ast \mathbf{j}}|\in \lbrack 2^{-q_{k}-1}|J_{%
\mathbf{i}}|,2^{-q_{k}}|J_{\mathbf{i}}|].
\end{equation*}%
Since $|J_{\mathbf{i}\ast \mathbf{j}}|\in \lbrack 2^{-q_{k}-1}|J_{\mathbf{i}%
}|,2^{-q_{k}}|J_{\mathbf{i}}|]$ for all $\mathbf{j}\in \mathcal{B}%
_{q_{k},k}\ $and int($J_{\mathbf{i}\ast \mathbf{j}}$)$\cap $int($J_{\mathbf{i%
}\ast \mathbf{j}^{\prime }})=\varnothing $ for all $\mathbf{j}\neq \mathbf{j}%
^{\prime }\in \mathcal{B}_{q_{k},k}$, by Lemma~\ref{lem2} again, there
exists a positive integer $l^{\prime }$ independent of $k$ such that each
ball with radius $r_{k}(\in \lbrack 2^{-q_{k}-1}|J_{\mathbf{i}%
}|,2^{-q_{k}}|J_{\mathbf{i}}|])$ intersects at most $l^{\prime }$ elements
in $\{J_{\mathbf{i}\ast \mathbf{j}}\}_{\mathbf{j}\in \mathcal{B}_{q_{k},k}}$.

To prove the lower bound, we need the following inequality
\begin{equation}
\frac{\sharp \mathcal{B}_{q_{k},k}}{l^{\prime }}\leq N_{r_{k},R_{k}}(E).
\label{lll}
\end{equation}
Notice that $J_{\mathbf{i}}\subset B(z,R_{k})$ for all $z\in J_{\mathbf{i}},$
we assume that there exists a smallest number $t$ such that $B(z,R_{k})$ can
be covered by $t$ balls of radius $r_{k},$ e.g.,
\begin{equation*}
B(z,R_{k})\subset B(x_{1},r_{k})\cup \cdots \cup B(x_{t},r_{k}).
\end{equation*}
Notice that $t\leq N_{r_{k},R_{k}}(E)$ and
\begin{equation*}
\bigcup_{\mathbf{j}\in \mathcal{B}_{q_{k},k}}J_{\mathbf{i}\ast \mathbf{j}%
}\subset J_{\mathbf{i}}\subset B(z,R_{k})\subset B(x_{1},r_{k})\cup \cdots
\cup B(x_{t},r_{k}).
\end{equation*}%
Then for any $\mathbf{j}\in \mathcal{B}_{q_{k},k},$ there exists at least a
ball $B(x_{i},r_{k})$ $1\leq i\leq t$ such that $J_{\mathbf{i}\ast \mathbf{j}%
}\cap B(x_{i},r_{k})\neq \varnothing $, that is,%
\begin{equation*}
\mathcal{B}_{q_{k},k}\subset \bigcup\nolimits_{i=1}^{t}\{\mathbf{j}\in
\mathcal{B}_{q_{k},k}:J_{\mathbf{i}\ast \mathbf{j}}\cap B(x_{i},r_{k})\neq
\varnothing \}.
\end{equation*}%
Therefore, we have p
\begin{eqnarray*}
\sharp \mathcal{B}_{q_{k},k} &\leq &\sum\nolimits_{i=1}^{t}\sharp \{\mathbf{j%
}\in \mathcal{B}_{q_{k},k}:J_{\mathbf{i}\ast \mathbf{j}}\cap
B(x_{i},r_{k})\neq \varnothing \} \\
&\leq &t\cdot l^{\prime }\leq N_{r_{k},R_{k}}(E)\cdot l^{\prime },
\end{eqnarray*}%
which completes the proof of inequality (\ref{lll}).

For any $\zeta >0,$ there exists $C_{\zeta }$ such that for any $k,$%
\begin{equation}
N_{r_{k},R_{k}}(E)\leq C_{\zeta }(\frac{R_{k}}{r_{k}})^{\dim _{A}E+\zeta }
\label{assouad}
\end{equation}%
Therefore, using (\ref{q_k}), (\ref{lll})\ and (\ref{assouad}), we have
\begin{eqnarray*}
\frac{2^{q_{k}(s-\varepsilon )}(1-2^{-\varepsilon })}{l^{\prime }}\leq \frac{%
\sharp \mathcal{B}_{q_{k},k}}{l^{\prime }} &\leq &N_{r_{k},R_{k}}(E) \\
&\leq &C_{\zeta }(\frac{R_{k}}{r_{k}})^{\dim _{A}E+\zeta }\leq C_{\zeta
}(2^{q_{k}+1})^{\dim _{A}E+\zeta }.
\end{eqnarray*}%
Since $\lim_{k}q_{k}=\infty$, by letting $k\rightarrow \infty ,$ it gives
\begin{equation*}
\dim _{A}E+\zeta \geq s-\varepsilon.
\end{equation*}%
By taking $\varepsilon \rightarrow 0$ and $\zeta \rightarrow 0,$ we have $%
\dim _{A}E\geq s$ for all $s<s^{\ast \ast },$ and thus $\dim _{A}E\geq
s^{\ast \ast }.$
\end{proof}

\section{Assouad Dimension of Homogeneous Set}

In this section we will prove the dimension formula for homogeneous sets.

\begin{lemma}
\label{qq} Suppose that $K\subset X$ is doubling. Then
\begin{equation}
\dim _{A}K=\lim_{\rho \rightarrow 0}\left( \sup\limits_{R<\varepsilon }\frac{%
\log N_{\rho R,R}(K)}{-\log \rho }\right) ,  \label{tttt}
\end{equation}%
for all $\varepsilon <|K|.$
\end{lemma}

\begin{proof}
First, we prove that
\begin{equation}
\dim _{A}K=\overline{\lim_{\rho \rightarrow 0}}\left( \sup\limits_{R<|K|}%
\frac{\log N_{\rho R,R}(K)}{\log R-\log (\rho R)}\right) .  \label{def}
\end{equation}
Arbitrarily choose a real $\alpha$ such that $\alpha >\overline{%
\lim\limits_{\rho \rightarrow 0}}\left( \sup\limits_{R<|K|}\frac{\log
N_{\rho R,R}(K)}{\log R-\log (\rho R)}\right) , $ there exists $\delta \in
(0,1)$ such that, for all $\rho <\delta ,$ we have $\sup\limits_{R<|K|}\frac{%
\log N_{\rho R,R}(K)}{\log R-\log (\rho R)}<\alpha $, that is ,
\begin{equation*}
N_{r,R}(K)\leq (R/r)^{\alpha },
\end{equation*}%
for $0<r<\delta R<R<|K|$. On the other hand, there exits a constant $%
c_{\delta }>0$ such that $N_{r,R}(K)\leq c_{\delta }$, for $\delta R<r<R$.
Hence, for all $0<r<R<|K|$, we have
\begin{equation*}
N_{r,R}(K)\leq c_{\delta }(R/r)^{\alpha },
\end{equation*}%
which implies that $\alpha \geq \dim _{A}K.$ Since $\alpha$ is arbitrarily
chosen, we have
\begin{equation*}
\overline{\lim\limits_{\rho \rightarrow 0}}\left( \sup\limits_{R<|K|}\frac{%
\log N_{\rho R,R}(K)}{\log R-\log (\rho R)}\right) \geq \dim _{A}K.
\end{equation*}

Suppose that $\alpha $ is a fixed number such that, for all $R<b$,
\begin{equation*}
N_{\rho R,R}(K)\leq c(R/\rho R)^{\alpha },
\end{equation*}%
where $b$ is a constant. Then%
\begin{equation*}
\sup\limits_{R<b}\frac{\log N_{\rho R,R}(K)}{\log R-\log (\rho R)}\leq \frac{%
\log c}{\log R-\log (\rho R)}+\alpha .
\end{equation*}%
Taking limit on both sides, we have $\overline{\lim\limits_{\rho \rightarrow
0}}\left( \sup\limits_{R<b}\frac{\log N_{\rho R,R}(K)}{\log R-\log (\rho R)}%
\right) \leq \alpha .$ Using the doubling property of $K,$ we have
\begin{equation*}
N_{\rho R,R}(K)\leq N_{\rho b/2,b/2}(K)\cdot N_{b/2,|K|}(K),
\end{equation*}%
for $b\leq R<|K|$, which implies%
\begin{equation*}
\overline{\lim\limits_{\rho \rightarrow 0}}\left( \sup\limits_{R<|K|}\frac{%
\log N_{\rho R,R}(K)}{\log R-\log (\rho R)}\right) =\overline{%
\lim\limits_{\rho \rightarrow 0}}\left( \sup\limits_{R<b}\frac{\log N_{\rho
R,R}(K)}{\log R-\log (\rho R)}\right) \leq \alpha .
\end{equation*}%
Since the inequality holds for all $\alpha >\dim _{A}K,$ it follows that
\begin{equation*}
\overline{\lim\limits_{\rho \rightarrow 0}}\left( \sup\limits_{R<|K|}\frac{%
\log N_{\rho R,R}(K)}{\log R-\log (\rho R)}\right) \leq \dim _{A}K,
\end{equation*}%
which finishes the proof of (\ref{def}).

We write
\begin{equation*}
t(\rho )=\sup\limits_{R}\frac{\log N_{\rho R,R}(K)}{-\log \rho }.
\end{equation*}%
To obtain the formula (\ref{tttt}), by (\ref{def}), it is sufficient to show
that the limit of $t(\rho )$ exists as $\rho$ tends to $0$.

Given $\rho>0 $. For any $\rho ^{\prime }<\rho ,$ there exists an integer $m$
such that
\begin{equation*}
\rho ^{m+1}\leq \rho ^{\prime }<\rho ^{m}.
\end{equation*}%
Since $N_{r_{1},r_{3}}(K)\leq N_{r_{1},r_{2}}(K)N_{r_{2},r_{3}}(K)$ for $%
r_{1}<r_{2}<r_{3},$ it follows that
\begin{equation*}
N_{(\rho ^{\prime })R,R}(K)\leq N_{\rho ^{m+1}R,R}(K)\leq \left(
\sup_{r}N_{\rho r,r}(K)\right) ^{m+1}.
\end{equation*}%
Hence, we have that
\begin{equation*}
\left\vert \frac{\log N_{(\rho ^{\prime })R,R}(K)}{\log \rho ^{\prime }}%
\right\vert \leq \left\vert \frac{\log \left( \sup_{r}N_{\rho r,r}(K)\right)
^{m+1}}{\log (\rho ^{\prime }/\rho ^{m+1})+(m+1)\log \rho }\right\vert ,
\end{equation*}%
and it implies%
\begin{equation*}
\underset{\rho ^{\prime }\rightarrow 0}{\overline{\lim }}t(\rho ^{\prime
})\leq \lim_{m\rightarrow \infty }\left\vert \frac{\log \left(
\sup_{r}N_{\rho r,r}(K)\right) ^{m+1}}{\log (\rho ^{\prime }/\rho
^{m+1})+(m+1)\log \rho }\right\vert =t(\rho )
\end{equation*}%
due to $1\leq \rho ^{\prime }/\rho ^{m+1}\leq \rho ^{-1}.$ Therefore, we
obtain that
\begin{equation*}
\underset{\rho ^{\prime }\rightarrow 0}{\overline{\lim }}t(\rho ^{\prime
})\leq \inf_{\rho }t(\rho )\leq \underset{\rho ^{\prime }\rightarrow 0}{%
\underline{\lim }}t(\rho ^{\prime }),
\end{equation*}%
that is,
\begin{equation*}
\lim_{\rho \rightarrow 0}t(\rho )=\inf_{\rho }t(\rho ).
\end{equation*}%
On the other hand, since $K$ is doubling, the
\begin{equation*}
\lim_{\rho \rightarrow 0}\left( \sup\limits_{R<\varepsilon _{1}}\frac{\log
N_{\rho R,R}(K)}{-\log \rho }\right) =\lim_{\rho \rightarrow 0}\left(
\sup\limits_{R<\varepsilon _{2}}\frac{\log N_{\rho R,R}(K)}{-\log \rho }%
\right) .
\end{equation*}
\end{proof}

\begin{proof}[Proof of Theorem 3]
\

Fix a point $x_{0}\in K.$ It is clear that $h\sim \alpha _{x_{0}}.$ By (\ref%
{E:equ}), we have that, for $r<R,$%
\begin{eqnarray*}
\left\vert \alpha _{x_{0}}(r)\log r-h\left( r\right) \log r\right\vert &\leq
&C, \\
\left\vert \alpha _{x_{0}}(R)\log R-h\left( R\right) \log R\right\vert &\leq
&C\text{.}
\end{eqnarray*}%
Hence%
\begin{eqnarray}
&&\left\vert \frac{{h\left( R\right) \log R-h\left( r\right) \log r}}{{\log
R-\log r}}-\frac{{\alpha _{x_{0}}(R)\log R-\alpha _{x_{0}}\left( r\right)
\log r}}{{\log R-\log r}}\right\vert  \notag \\
&\leq &\left\vert \frac{{\alpha _{x_{0}}(R)\log R-h\left( R\right) \log R}}{{%
\log R-\log r}}\right\vert +\left\vert \frac{{\alpha _{x_{0}}(r)\log
r-h\left( r\right) \log r}}{{\log R-\log r}}\right\vert  \label{307} \\
&\leq &\frac{2C}{|\log R/r|}.  \notag
\end{eqnarray}

Suppose $k$ is the smallest number of balls with radius $r$ needed to cover $%
B(x,R),$ i.e., suppose $B(x,R)$ is covered by $B(y_{1},r),\cdots
,B(y_{k},r). $ In fact, we can choose
\begin{equation}
k=N_{r,R}(K).  \label{wanz1}
\end{equation}%
Then
\begin{equation*}
\mu (B(x,R))\leq \sum\limits_{i=1}^{k}\mu (B(y_{i},r))
\end{equation*}%
which implies
\begin{equation*}
\frac{\mu (B(x,R))}{\max_{y\in K}\mu (B(y,r))}\leq k.
\end{equation*}%
Using (\ref{self1}), we have
\begin{equation}
\lambda ^{-2}\frac{\mu (B(x_{0},R))}{\mu (B(x_{0},r))}\leq k.  \label{wanz2}
\end{equation}

We also assume $p$ is the largest number of disjoint $(r/2)$-balls with
centers in $B(x,R),$ for example, $B(z_{1},r/2),\cdots ,B(z_{p},r/2)$ are
pairwise disjoint. By the routine argument, we have
\begin{equation*}
k\leq p.
\end{equation*}%
In the same way,
\begin{equation*}
p\min_{y\in K}\mu (B(y,r/2))\leq \sum\limits_{i=1}^{p}\mu (B(z_{i},r))\leq
\mu (B(x,R+r))\leq \mu (B(x,2R)).
\end{equation*}%
Therefore, using (\ref{self1}), we have
\begin{equation}
k\leq p\leq \frac{\mu (B(x,2R))}{\min_{y}\mu (B(y,r/2))}\leq \lambda ^{2}%
\frac{\mu (B(x_{0},2R))}{\mu (B(x_{0},r/2))}.  \label{wanz3}
\end{equation}

Using (\ref{***}), the measure $\mu $ is doubling, i.e., there is a constant
$D>0$ such that
\begin{eqnarray*}
\mu (B(x_{0},2R)) &\leq &D\mu (B(x_{0},R)), \\
\mu (B(x_{0},r/2)) &\geq &D^{-1}\mu (B(x_{0},r)).
\end{eqnarray*}%
Then (\ref{wanz3}) shows that
\begin{equation}
k\leq (\lambda D)^{2}\frac{\mu (B(x_{0},R))}{\mu (B(x_{0},r))}.
\label{wanz4}
\end{equation}%
Combining (\ref{wanz1}), (\ref{wanz2}) and (\ref{wanz4}), we obtain that
\begin{eqnarray}
&&\frac{\log \lambda ^{-2}}{\log R-\log r}+\frac{\alpha _{x_{0}}(R)\log
R-\alpha _{x_{0}}(r)\log r}{\log R-\log r}  \notag \\
&\leq &\frac{\log N_{r,R}(K)}{\log R-\log r}  \label{306} \\
&\leq &\frac{\log (\lambda D)^{2}}{\log R-\log r}+\frac{\alpha
_{x_{0}}(R)\log R-\alpha _{x_{0}}(r)\log r}{\log R-\log r}.  \notag
\end{eqnarray}%
By Lemma \ref{qq}, (\ref{307}) and (\ref{306}), we obtain that
\begin{equation*}
\dim _{A}K=\lim_{\rho \rightarrow 0}\left( \sup_{R}\frac{h(R)\log R-h(\rho
R)\log (\rho R)}{-\log \rho }\right) .
\end{equation*}
\end{proof}

%\bigskip


\begin{thebibliography}{99}
\bibitem{Assouad1} P. Assouad, Espaces m\'{e}triques, plongements, facteurs,
Th\`ese de doctorat, Publ. Math. Orsay No. 223-7769, Univ. Paris XI, Orsay,
1977.

\bibitem{Assouad2} P. Assouad, \`{E}tude d'une dimension m\'{e}trique li\'{e}%
e \`{a} la possibilit\'{e} de plongements dans $\mathbf{R}^{n}$, C.~R. Acad.
Sci. Paris S\'{e}r. A-B, 288(15), 1979, 731--734.

\bibitem{Assouad3} P. Assouad, Pseudodistances, facteurs et dimension m\'{e}%
trique, in Seminaire D'Analyse Harmonique (1979-1980), pp. 1--33, Publ.
Math. Orsay 80, 7, Univ. Paris XI, Orsay, 1980.

\bibitem{Bed} T. Bedford, Crinkly curves, Markov partitions and box
dimension in self-similar sets, PhD Thesis, University of Warwick, 1984.

\bibitem{Edgar} G. Edgar, Integral, Probability, and Fractal Measures,
Springer-Verlag, New York, 1998.

\bibitem{Falconer} K J. Falconer, Fractal Geometry. Mathematical Foundations
and Applications, John Wiley Sons, Ltd., Chichester, 1990.

\bibitem{Fraser} J. M. Fraser, Assouad type dimensions and homogeneity of
fractals, arXiv:1301.2934, 2013.

\bibitem{Heinonen} J. Heinonen, Lectures on Analysis on Metric Spaces,
Springer-Verlag, New York, 2001.

\bibitem{Hua} S. Hua, On the Hausdorff dimension of generalized self-similar
sets (Chinese), Acta Math. Appl. Sinica, 17(4), 1994, 551--558.

\bibitem{HuaLi} S. Hua, W. X. Li, Packing dimension of generalized Moran
Sets, Progr. Natur. Sci. (English Ed.), 6(2), 1996, 148--152.

\bibitem{LG} S. Lalley, D Gatzouras, Hausdorff and box dimensions of certain
self-affine fractals. Indiana Univ. Math. J. 41(2), 1992, 533--568.

\bibitem{Lehrb} J. Lehrb, H. Tuominen, A note on the dimensions of Assouad
and Aikawa, J. Math. Soc. Japan, 65(2), 2013, 343--356.

\bibitem{Li} J. J. Li, Assouad dimensions of Moran sets, C. R. Math. Acad.
Sci. Paris, 351(1-2), 2013, 19--22.

\bibitem{Luukkainen} J. Luukkainen, Assouad dimension: Antifractal
metrization, porous sets, and homogeneous measures, J. Korean Math. Soc.,
35, 1998, 23--76.

\bibitem{Lv} F. L\"{u}, M. L. Lou, Z. Y. Wen, L. F. Xi, Bilipschitz
embedding of homogeneous set, arXiv:1402.0080, 2014.

\bibitem{Mackay} J. M. Mackay, Assouad dimension of self-affine carpets,
Conform. Geom. Dyn., 15, 2011, 177--187.

\bibitem{Mattila} P. Mattila, Geometry of Sets and Measure in Euclidean
Spaces, Cambridge University Press, Cambridge, 1995.

\bibitem{Mcm} C. McMullen, The Hausdorff dimension of general Sierpinski
carpets, Nagoya Math. J. 96, 1-9, 1984.

\bibitem{Moran} P. A. Moran, Additive functions of intervals and Hausdorff
measure, Proc. Camb. Phil. Soc., 42, 1946, 15--23.

\bibitem{Olsen} L. Olsen, On the Assouad dimension of graph directed Moran
fractals, Fractals, 19, 2011, 221--226.

\bibitem{Tricot} C. Tricot, Curves and Fractal Dimension, Springer-Verlag,
New York, 1995.

\bibitem{Wen1} Z. Y. Wen, Mathematical Foundations of Fractal Geometry,
Shanghai Scientific and Tech- nological Education Publishing House,
Shanghai, 2000.

\bibitem{Wen2} Z. Y. Wen, Moran sets and Moran classes, Chinese Sci. Bull.,
46, 2001, 1849--1856.
\end{thebibliography}
\end{document}